\documentclass[12pt,reqno]{amsart}

\usepackage{amssymb,amsmath}
\usepackage{amsfonts}
\usepackage{amsthm}
\usepackage{mathrsfs}
\usepackage{graphicx}
\usepackage[T1]{fontenc}% Brzd\k{e}k
\theoremstyle{plain} 
\newtheorem{theorem}{Theorem}
\newtheorem{corollary}[theorem]{Corollary}

\newtheorem{lemma}[theorem]{Lemma}

\theoremstyle{definition} 
\newtheorem{definition}[theorem]{Definition}
\newtheorem{remark}[theorem]{Remark}
\newtheorem{example}[theorem]{Example}

\newcommand{\R}{\ensuremath{\mathbb{R}}}
\newcommand{\Z}{\ensuremath{\mathbb{Z}}}
\newcommand{\T}{\ensuremath{\mathbb{T}}}
\newcommand{\N}{\ensuremath{\mathbb{N}}}
\newcommand{\C}{\ensuremath{\mathbb{C}}}
\DeclareMathOperator{\Log}{Log}

\numberwithin{equation}{section}
\numberwithin{theorem}{section}

\small\normalsize

\begin{document}

\title{A Multi-Valued Logarithm on Time Scales}
\author[Anderson]{Douglas R. Anderson} 
\address{Department of Mathematics \\
         Concordia College \\
         Moorhead, Minnesota 56562 USA}
\email{andersod@cord.edu}
\author[Bohner]{Martin Bohner}
\address{Department of Mathematics and Statistics \\
Missouri University of Science and Technology \\
Rolla, Missouri 65409-0020 USA}
\email{bohner@mst.edu}


\begin{abstract} 
A new definition of a multi-valued logarithm on time scales is introduced for delta-differentiable functions that never vanish. 
This new logarithm arises naturally from the definition of the cylinder transformation that is also at the heart of the definition of exponential functions on time scales. This definition will lead to a logarithm function on arbitrary time scales with familiar and useful properties that previous definitions in the literature lacked.
\end{abstract}

\maketitle\thispagestyle{empty}

\begin{center}
Dedicated to Professor Allan C. Peterson, our mentor, colleague, and friend, on the occasion of his retirement after 51 years at the University of Nebraska-Lincoln.
\end{center}

%%%%%%%%%%%%%%%%% 
%               % 
% SECTION Intro % 
%               % 
%%%%%%%%%%%%%%%%%

\section{introduction}

A recurring open problem for dynamic equations on time scales \cite{bp,bp2} has been the following \cite{bohner}: Define a ``nice'' logarithm function on time scales and present its properties. The aim of this paper is to introduce on time scales a new multi-valued logarithm arising from the cylinder transformation employed in definitions of exponential functions for dynamic equations. 

The development of this logarithm on general time scales will proceed as follows. In Section \ref{newlog}, we extend the definition of the traditional single-valued cylinder transformation to a multi-valued cylinder transformation. This transformation has useful properties across the circle plus $(\oplus)$ and circle dot $(\odot)$ operations, and is the basis for the definition of the logarithm, for non-vanishing delta-differentiable functions. In Section \ref{logproperties}, nice properties of this new logarithm are shown to hold. Section \ref{nablacase} establishes a similar logarithm for the nabla case. In Section \ref{SectionCayley}, the Cayley cylinder transformation is also considered, and is shown to lead to the very same logarithm. Finally, in Section \ref{historylog}, we give a brief list of logarithm functions on time scales from the literature.
Numerous examples on various time scales are given throughout the paper to illustrate our results.

Throughout this paper, we will assume a working knowledge of time scales and time scales notation.

%%%%%%%%%%%%%%%%% 
%               % 
%  SECTION Log  % 
%               % 
%%%%%%%%%%%%%%%%%

\section{a new logarithm on time scales}\label{newlog}

We begin our presentation of a new definition of a logarithm for dynamic equations on time scales with some motivation provided by the definition of exponential functions for dynamic equations based on the cylinder transformation. The following definition \cite[Definition 2.21]{bp} is the original cylinder transformation; a modified cylinder transformation will also be examined, in Section \ref{SectionCayley}.

% Definition: Cylinder Transformation %

\begin{definition}[Single-Valued Cylinder Transformation]
For $h>0$, define the cylinder transformation $\xi_h:\C_h\rightarrow\Z_h$ by
\begin{equation}\label{single-cylinder}
 \xi_h(z) = \begin{cases} \displaystyle\frac{1}{h}\Log(1+zh) &\text{for}\; h\ne 0 \\ z &\text{for}\; h=0, \end{cases} 
\end{equation}
where $\C$ is the set of complex numbers, 
\begin{equation}\label{chzhdefs}
 \C_h=\left\{z\in\C:\; z\ne -\frac{1}{h}\right\}, \qquad \Z_h=\left\{z\in\C:\; -\frac{\pi}{h}<\operatorname{Im}(z)\le \frac{\pi}{h}\right\},
\end{equation} 
and $\Log$ is the principal logarithm function.
\end{definition}

The following definition is \cite[Definition 2.25]{bp}.

% Definition: Regressive Functions %

\begin{definition}[Regressive Function]
A function $p:\T\rightarrow\R$ is regressive provided
\[ 1+\mu(t)p(t) \ne 0 \quad\text{for all}\quad t\in\T^{\kappa} \]
holds. The set of all regressive and rd-continuous functions $p:\T\rightarrow\R$ is denoted by $\mathcal{R}$.
\end{definition}

The following definition is \cite[Definition 2.30]{bp}.

% Definition: Exponential Function %

\begin{definition}[Exponential Function]
For functions $p\in\mathcal{R}$, the exponential function on time scales is given by
\[ e_p(t,s)=\exp\left(\int_s^t \xi_{\mu(\tau)}(p(\tau))\Delta\tau\right) \quad\text{for}\quad s,t\in\T, \]
where $\xi_h(z)$ is the cylinder transformation given in \eqref{single-cylinder}.
\end{definition}

We now set the foundation for offering a new definition of logarithms on time scales. This definition will be of a multi-valued function, for which we need to modify the single-valued cylinder function given in \eqref{single-cylinder}. 

% Definition: Cylinder Transformation %

\begin{definition}[Multi-Valued Cylinder Transformation]
For $h>0$, define the multi-valued cylinder transformation $\zeta_h:\C_h\rightarrow\C$ by
\begin{equation}\label{cylinder}
 \zeta_h(z) = \begin{cases} \displaystyle\frac{1}{h}\log(1+zh) &\text{for}\; h\ne 0 \\ z &\text{for}\; h=0, \end{cases} 
\end{equation}
where $\C$ is the set of complex numbers, $\C_h$ is given in \eqref{chzhdefs}, and $\log$ is the multi-valued complex logarithm function.
\end{definition}

% Lemma: Multi-Cylinder Properties %

\begin{lemma}\label{ximuprop}
Let $f,g:\T\rightarrow\C$ be $\Delta$-differentiable functions with $f,g\ne 0$ on $\T$, and let the multi-valued cylinder transformation $\zeta$ be given by \eqref{cylinder}. Then, for fixed $\tau\in\T^{\kappa}$,
%\begin{equation}
$$ \zeta_{\mu(\tau)}\left(\left(\frac{f^{\Delta}}{f}\oplus \frac{g^{\Delta}}{g}\right)(\tau)\right) = \zeta_{\mu(\tau)}\left(\frac{f^{\Delta}(\tau)}{f(\tau)}\right) + \zeta_{\mu(\tau)}\left(\frac{g^{\Delta}(\tau)}{g(\tau)}\right). $$
%\end{equation}
\end{lemma}

\begin{proof}
First, note that the simple useful formula $f^{\sigma}=f+\mu f^{\Delta}$ (suppressing the variable) implies 
\begin{eqnarray*}
\frac{(fg)^{\Delta}}{fg} 
&=& \frac{f^{\sigma}g^{\Delta}+f^{\Delta}g}{fg} \\
&=& \frac{(f+\mu f^{\Delta})g^{\Delta}}{fg} + \frac{f^{\Delta}}{f} \\
&=& \frac{f^{\Delta}}{f} + \frac{g^{\Delta}}{g} + \mu\frac{f^{\Delta} g^{\Delta}}{fg} \\
&=& \frac{f^{\Delta}}{f}\oplus \frac{g^{\Delta}}{g}.
\end{eqnarray*}
It follows that for fixed $\tau\in\T^{\kappa}$,
\begin{eqnarray*}
\lefteqn{\zeta_{\mu(\tau)}\left(\left(\frac{f^{\Delta}}{f}\oplus \frac{g^{\Delta}}{g}\right)(\tau)\right)} \\
&=& \zeta_{\mu(\tau)}\left(\frac{(fg)^{\Delta}(\tau)}{(fg)(\tau)}\right) \\
&=& \begin{cases} \displaystyle\frac{1}{\mu(\tau)}\log\left(1+\mu(\tau)\frac{(fg)^{\Delta}(\tau)}{(fg)(\tau)}\right) &\text{for}\; \mu(\tau)\ne 0 \\ \frac{(fg)^{\Delta}(\tau)}{(fg)(\tau)} &\text{for}\; \mu(\tau)=0 \end{cases} \\
&=& \begin{cases} \displaystyle\frac{1}{\mu(\tau)}\log\left(\frac{(fg)^{\sigma}(\tau)}{(fg)(\tau)}\right) &\text{for}\; \mu(\tau)\ne 0 \\ \left(\frac{f^{\Delta}}{f}\oplus \frac{g^{\Delta}}{g}\right)(\tau) &\text{for}\; \mu(\tau)=0 \end{cases} \\
&=& \begin{cases} \displaystyle\frac{1}{\mu(\tau)}\log\left(\frac{f^{\sigma}(\tau)}{f(\tau)}\right) + \displaystyle\frac{1}{\mu(\tau)}\log\left(\frac{g^{\sigma}(\tau)}{g(\tau)}\right) &\text{for}\; \mu(\tau)\ne 0 \\ \left(\frac{f^{\Delta}}{f} + \frac{g^{\Delta}}{g}\right)(\tau) &\text{for}\; \mu(\tau)=0 \end{cases} \\
&=& \begin{cases} \displaystyle\frac{1}{\mu(\tau)}\log\left(\frac{(f+\mu f^{\Delta})(\tau)}{f(\tau)}\right) + \displaystyle\frac{1}{\mu(\tau)}\log\left(\frac{(g+\mu g^{\Delta})(\tau)}{g(\tau)}\right) &\text{for}\; \mu(\tau)\ne 0 \\ \frac{f^{\Delta}(\tau)}{f(\tau)} + \frac{g^{\Delta}(\tau)}{g(\tau)} &\text{for}\; \mu(\tau)=0 \end{cases} \\
&=& \zeta_{\mu(\tau)}\left(\frac{f^{\Delta}(\tau)}{f(\tau)}\right) + \zeta_{\mu(\tau)}\left(\frac{g^{\Delta}(\tau)}{g(\tau)}\right).
\end{eqnarray*}
This completes the proof.
\end{proof}

% Lemma: Multi-Cylinder Exponent Property %

\begin{lemma}\label{zetamuprop}
Let $\alpha\in\R$, and let $p:\T\rightarrow\C$ be a $\Delta$-differentiable function with $p\ne 0$ on $\T$. For the multi-valued cylinder transformation $\zeta$ given by \eqref{cylinder} and for fixed $\tau\in\T^{\kappa}$,
%\begin{equation}
$$ \zeta_{\mu(\tau)}\left(\left(\alpha\odot\frac{p^{\Delta}}{p}\right)(\tau)\right) = \alpha\zeta_{\mu(\tau)}\left(\frac{p^{\Delta}(\tau)}{p(\tau)}\right). $$
%\end{equation}
\end{lemma}

\begin{proof}
Let $\alpha\in\R$, and let $p:\T\rightarrow\C$ be a $\Delta$-differentiable function with $p\ne 0$ on $\T$. Then \cite[Theorem 2.43]{bp2} yields
$$ 1+\mu(\alpha\odot f) = (1+\mu f)^{\alpha} $$
on $\T^{\kappa}$ for $f=\frac{p^{\Delta}}{p}$. It follows that for fixed $\tau\in\T^{\kappa}$,
\begin{eqnarray*}
\lefteqn{\zeta_{\mu(\tau)}\left(\left(\alpha\odot\frac{p^{\Delta}}{p}\right)(\tau)\right)} \\
&=& \begin{cases} \displaystyle\frac{1}{\mu(\tau)}\log\left(1+\mu(\tau)\left(\alpha\odot\frac{p^{\Delta}}{p}\right)(\tau)\right) &\text{for}\; \mu(\tau)\ne 0 \\ \left(\alpha\odot\frac{p^{\Delta}}{p}\right)(\tau) &\text{for}\; \mu(\tau)=0 \end{cases} \\
&=& \begin{cases} \displaystyle\frac{1}{\mu(\tau)}\log\left(1+\mu(\tau)\frac{p^{\Delta}(\tau)}{p(\tau)}\right)^{\alpha} &\text{for}\; \mu(\tau)\ne 0 \\ \alpha\frac{p^{\Delta}(\tau)}{p(\tau)} &\text{for}\; \mu(\tau)=0 \end{cases} \\
&=& \alpha\begin{cases} \displaystyle\frac{1}{\mu(\tau)}\log\left(1+\mu(\tau)\frac{p^{\Delta}(\tau)}{p(\tau)}\right) &\text{for}\; \mu(\tau)\ne 0 \\ \frac{p^{\Delta}(\tau)}{p(\tau)} &\text{for}\; \mu(\tau)=0 \end{cases} \\
&=& \alpha\zeta_{\mu(\tau)}\left(\frac{p^{\Delta}(\tau)}{p(\tau)}\right).
\end{eqnarray*}
This completes the proof.
\end{proof}

% Definition: Logarithm Function %

\begin{definition}[Logarithm Function]
For a $\Delta$-differentiable function $p:\T\rightarrow\C$ with $p\ne 0$ on $\T$, the multi-valued logarithm function on time scales is given by
\[ \ell_p(t,s) = \int_s^t \zeta_{\mu(\tau)}\left(\frac{p^{\Delta}(\tau)}{p(\tau)}\right)\Delta\tau \quad\text{for}\quad s,t\in\T, \]
where $\zeta_h(z)$ is the multi-valued cylinder transformation given in \eqref{cylinder}. Define the principal logarithm on time scales to be
\[ L_p(t,s) = \int_s^t \xi_{\mu(\tau)}\left(\frac{p^{\Delta}(\tau)}{p(\tau)}\right)\Delta\tau \quad\text{for}\quad s,t\in\T, \]
where $\xi_h(z)$ is the single-valued cylinder transformation given in \eqref{single-cylinder}.
\end{definition}

% Remark: no constant %

\begin{remark}
According to this definition, if $p\equiv\;$constant, then $\ell_p(t,s)=0$ for all $t,s\in\T$. Thus, this logarithm does not distinguish between either constants or constant multiples of functions. We moreover note here that even in the case $\T=\R$, the dynamics along the negative and positive real line necessitate the existence of a logarithm with principal and multiple values, making a multi-valued logarithm on general time scales both natural and expected, though heretofore unexplored.
\end{remark}

% Example R, hZ, q^N %

\begin{example}
For $\T=\R$, 
\[ \ell_p(t,s) = \int_s^t \zeta_{\mu(\tau)}\left(\frac{p^{\Delta}(\tau)}{p(\tau)}\right)\Delta\tau = \int_s^t \frac{p'(\tau)}{p(\tau)}{\rm d}\tau = \log\left(\frac{p(t)}{p(s)}\right), \]
where $\log$ is the multi-valued complex logarithm function.
For $\T=h\Z$, 
$$f^{\Delta}(\tau)=\Delta_hf(\tau)=\frac{f(\tau+h)-f(\tau)}{h}$$
and
\begin{eqnarray*}
\ell_p(t,s) 
&=& \int_s^t \zeta_{\mu(\tau)}\left(\frac{p^{\Delta}(\tau)}{p(\tau)}\right)\Delta\tau = \sum_{j=\frac{s}{h}}^{\frac{t}{h}-1} \zeta_{h}\left(\frac{\Delta_h p(jh)}{p(jh)}\right)h = \sum_{j=\frac{s}{h}}^{\frac{t}{h}-1} \frac{1}{h}\log\left(1+\frac{h\Delta_h p(jh)}{p(jh)}\right)h \\
&=& \sum_{j=\frac{s}{h}}^{\frac{t}{h}-1} \log\left(\frac{p(jh+h)}{p(jh)}\right) = \log\left(\prod_{j=\frac{s}{h}}^{\frac{t}{h}-1} \frac{p((j+1)h)}{p(jh)}\right) = \log\left(\frac{p(t)}{p(s)}\right).
\end{eqnarray*}
For $\T=q^{\N_0}$, 
$$f^{\Delta}(\tau)=\frac{f(q\tau)-f(\tau)}{(q-1)\tau}$$
and
\begin{eqnarray*}
	\ell_p(t,s) 
	&=& \int_s^t \zeta_{\mu(\tau)}\left(\frac{p^{\Delta}(\tau)}{p(\tau)}\right)\Delta\tau \\
	&=& \sum_{\tau\in[s,t)} \zeta_{(q-1)\tau}\left(\frac{p^{\Delta}(\tau)}{p(\tau)}\right)(q-1)\tau \\
	&=& \sum_{\tau\in[s,t)} \frac{1}{(q-1)\tau}\log\left(1+\frac{(q-1)\tau  p^{\Delta}(\tau)}{p(\tau)}\right)(q-1)\tau \\
	&=& \sum_{\tau\in[s,t)} \log\left(\frac{p(q\tau)}{p(\tau)}\right) \\
	&=& \log\left(\frac{p(t)}{p(s)}\right).
\end{eqnarray*}
This ends the example.
\end{example}

% Example [a,b]u[c,d] %

\begin{example}
For $a,b,c,d\in\R$ with $a<b<c<d$, let $\T=[a,b]\cup[c,d]$. Assume $p:\T\rightarrow\C$ is differentiable with $p\ne 0$ on $\T$. If $s,t\in[a,b)$ or $s,t\in[c,d]$, then $\mu(\tau)\equiv 0$ for $\tau\in[s,t]$, so that by the definition of the multi-valued cylinder function \eqref{cylinder},
$$ \ell_p(t,s) = \int_s^t \frac{p'(\tau)}{p(\tau)}{\rm d}\tau = \log\left(\frac{p(t)}{p(s)}\right). $$
Assume without loss of generality that $s\in[a,b]$ and $t\in[c,d]$. Then $c=\sigma(b)$, and
\begin{eqnarray*}
  \ell_p(t,s) &=& \int_s^t \zeta_{\mu(\tau)}\left(\frac{p^{\Delta}(\tau)}{p(\tau)}\right)\Delta\tau \\
	&=& \left(\int_s^b + \int_b^{\sigma(b)} + \int_{\sigma(b)}^t \right) \zeta_{\mu(\tau)}\left(\frac{p^{\Delta}(\tau)}{p(\tau)}\right)\Delta\tau \\
	&=& \log\left(\frac{p(b)}{p(s)}\right) + \log\left(\frac{p(t)}{p(\sigma(b))}\right) + \int_b^{\sigma(b)}\zeta_{\mu(\tau)}\left(\frac{p^{\Delta}(\tau)}{p(\tau)}\right)\Delta\tau \\
	&=& \log\left(\frac{p(b)}{p(s)}\right) + \log\left(\frac{p(t)}{p(c)}\right) + \mu(b)\zeta_{\mu(b)}\left(\frac{p^{\Delta}(b)}{p(b)}\right) \\
	&=& \log\left(\frac{p(b)}{p(s)}\right) + \log\left(\frac{p(t)}{p(c)}\right) + \mu(b)\left[\frac{1}{\mu(b)}\log\left(1+\frac{\mu(b)p^{\Delta}(b)}{p(b)}\right)\right] \\
	&=& \log\left(\frac{p(b)}{p(s)}\right) + \log\left(\frac{p(t)}{p(c)}\right) + \log\left(\frac{p^{\sigma}(b)}{p(b)}\right) \\
	&=& \log\left(\frac{p(b)}{p(s)}\right) + \log\left(\frac{p(t)}{p(c)}\right) + \log\left(\frac{p(c)}{p(b)}\right) \\
	&=& \log\left(\frac{p(t)}{p(s)}\right).
\end{eqnarray*}
Consequently, in all cases, we see that $\ell_p(t,s) = \log\left(\frac{p(t)}{p(s)}\right)$ on this time scale as well. 
\end{example}

% Example: Principal Log %

\begin{example}
Let $\T=(-\infty,-4]\cup[2,\infty)$, and $p(t)=t^3$. Let $t=3$ and $s=-5$. Then 
$$ \mu(-4) = \sigma(-4)-(-4) = 2-(-4) = 6, $$ 
and the principal logarithm on this time scale is
\begin{eqnarray*}
L_p(t,s) &=& L_{p}(3,-5) \\
&=& \int_{-5}^{3} \xi_{\mu(\tau)}\left(\frac{(\tau^3)^{\Delta}}{\tau^3}\right)\Delta\tau \\
&=& \left(\int_{-5}^{-4}+\int_{-4}^{2}+\int_{2}^{3}\right)\xi_{\mu(\tau)}\left(\frac{\sigma(\tau)^2+\tau\sigma(\tau)+\tau^2}{\tau^3}\right)\Delta\tau \\
&=& 3\left(\int_{-5}^{-4}+\int_{2}^{3}\right) \frac{{\rm d}\tau}{\tau} +\mu(-4) \xi_{\mu(-4)}\left(\frac{2^2-4(2)+(-4)^2}{(-4)^3}\right) \\
&=& 3\left(\Log[-4]-\Log[-5]+\Log[3]-\Log[2]\right) + \Log\left(1+6\frac{12}{-64}\right) \\
&=& \Log\left(\frac{27}{-125}\right) \\
&=& \ln\left(\frac{27}{125}\right)+i\pi,
\end{eqnarray*}
where $\Log$ is the principal complex logarithm, and $\ln$ is the natural logarithm.
\end{example}

%%%%%%%%%%%%%%%%%%%%%%%
% Section: Properties %
%%%%%%%%%%%%%%%%%%%%%%%

\section{properties of the logarithm}\label{logproperties}

Using the definition of the multi-valued logarithm on time scales given above, we establish the following properties. 

% Theorem: Inverse of sorts %

\begin{theorem}
Let $p:\T\rightarrow\C$ be a $\Delta$-differentiable function with $p\ne 0$ on $\T$. Then, for $s,t\in\T$, we have
$$\exp\left(L_{p}(t,s)\right)=e_{\frac{p^{\Delta}}{p}}(t,s).$$
\end{theorem}

\begin{proof}
Let $p:\T\rightarrow\C$ be a $\Delta$-differentiable function with $p\ne 0$ on $\T$. Then, for $s,t\in\T$, we have
\begin{eqnarray*}
L_{p}(t,s) &=& \int_s^t \xi_{\mu(\tau)}\left(\frac{p^{\Delta}(\tau)}{p(\tau)}\right)\Delta\tau.
\end{eqnarray*}
Now exponentiate both sides and use the definition of the exponential function $e_p(t,s)$.
\end{proof}

% Corollary %

\begin{corollary}
Let $p\in\mathcal{R}$ and $s,t\in\T$. Then
$$\exp\left(L_{e_p}(t,s)\right)=e_p(t,s).$$
\end{corollary}

% Theorem: Product Rule Quotient Rule %

\begin{theorem}[Logarithm of Product \& Quotient]
Let $f,g:\T\rightarrow\C$ be $\Delta$-differentiable functions with $f,g\ne 0$ on $\T$. Then, for $s,t\in\T$, we have
$$ \ell_{fg}(t,s) =  \ell_{f}(t,s) +  \ell_{g}(t,s) $$
and
$$ \ell_{\frac{f}{g}}(t,s) =  \ell_{f}(t,s) -  \ell_{g}(t,s). $$
\end{theorem}

\begin{proof}
Let $f,g:\T\rightarrow\R$ be $\Delta$-differentiable functions with $f,g\ne 0$ on $\T$. Then, for $s,t\in\T$, we have via Lemma \ref{ximuprop} and its proof that
\begin{eqnarray*}
\ell_{fg}(t,s) &=& \int_s^t \zeta_{\mu(\tau)}\left(\frac{(fg)^{\Delta}(\tau)}{(fg)(\tau)}\right)\Delta\tau \\
&=&  \int_s^t \zeta_{\mu(\tau)}\left(\left(\frac{f^{\Delta}}{f}\oplus \frac{g^{\Delta}}{g}\right)(\tau)\right)\Delta\tau \\
&=&  \int_s^t \zeta_{\mu}\left(\frac{f^{\Delta}(\tau)}{f(\tau)}\right)\Delta\tau +  \int_s^t \zeta_{\mu}\left(\frac{g^{\Delta}(\tau)}{g(\tau)}\right)\Delta\tau \\
&=& \ell_{f}(t,s) + \ell_{g}(t,s).
\end{eqnarray*}
In a similar manner,
\begin{eqnarray*}
\ell_{\frac{f}{g}}(t,s) &=& \int_s^t \zeta_{\mu(\tau)}\left(\frac{\left(\frac{f}{g}\right)^{\Delta}(\tau)}{\left(\frac{f}{g}\right)(\tau)}\right)\Delta\tau \\
&=&  \int_s^t \zeta_{\mu(\tau)}\left(\left(\frac{f^{\Delta}}{f}\ominus \frac{g^{\Delta}}{g}\right)(\tau)\right)\Delta\tau \\
&=&  \int_s^t \zeta_{\mu}\left(\frac{f^{\Delta}(\tau)}{f(\tau)}\right)\Delta\tau -  \int_s^t \zeta_{\mu}\left(\frac{g^{\Delta}(\tau)}{g(\tau)}\right)\Delta\tau \\
&=& \ell_{f}(t,s) - \ell_{g}(t,s).
\end{eqnarray*}
This completes the proof.
\end{proof}

% Theorem: Power Rule %

\begin{theorem}
Let $\alpha\in\R$, and let $p:\T\rightarrow\C$ be a $\Delta$-differentiable function with $p\ne 0$ on $\T$. Then, for $s,t\in\T$, we have
$$ \ell_{p^{\alpha}}(t,s) =  \alpha\ell_{p}(t,s). $$
\end{theorem}

\begin{proof}
For the multi-valued cylinder transformation $\zeta$ given by \eqref{cylinder} and for fixed $\tau\in\T^{\kappa}$,
$$ \zeta_{\mu(\tau)}\left(\left(\alpha\odot\frac{p^{\Delta}}{p}\right)(\tau)\right) = \alpha\zeta_{\mu(\tau)}\left(\frac{p^{\Delta}(\tau)}{p(\tau)}\right) $$
using Lemma \ref{zetamuprop}. Moreover, by \cite[Theorem 2.37]{bp2}, we have
$$ \frac{\left(p^{\alpha}\right)^{\Delta}}{p^{\alpha}} = \alpha\odot \frac{p^{\Delta}}{p}. $$
Consequently,
\begin{eqnarray*}
\ell_{p^{\alpha}}(t,s) &=& \int_s^t \zeta_{\mu(\tau)}\left(\frac{\left(p^{\alpha}\right)^{\Delta}(\tau)}{p^{\alpha}(\tau)}\right)\Delta\tau \\
&=& \int_s^t \zeta_{\mu(\tau)}\left(\left(\alpha\odot \frac{p^{\Delta}}{p}\right)(\tau)\right)\Delta\tau \\
&=& \int_s^t \alpha\zeta_{\mu(\tau)}\left(\frac{p^{\Delta}(\tau)}{p(\tau)}\right)\Delta\tau \\
&=& \alpha\ell_{p}(t,s).
\end{eqnarray*}
This ends the proof.
\end{proof}

% Theorem: Derivative of Log %

\begin{theorem}
Let $p:\T\rightarrow\R$ be a $\Delta$-differentiable function with $p\ne 0$ on $\T$. Then, for $s,t\in\T$, we have
$$ \ell_{p}^{\Delta}(t,s)=\begin{cases} \frac{1}{\mu(t)}\log\left(\frac{p^{\sigma}(t)}{p(t)}\right) &\text{for}\; \mu(t)\ne 0 \\ \frac{p^{\Delta}(t)}{p(t)} &\text{for}\; \mu(t)=0, \end{cases} $$
where $\Delta$-differentiation is with respect to $t$.
\end{theorem}

\begin{proof}
Using the definition of the logarithm and $\Delta$-differentiating with respect to $t$,
\begin{eqnarray*}
\ell_{p}^{\Delta}(t,s) &=& \zeta_{\mu(t)}\left(\frac{p^{\Delta}(t)}{p(t)}\right) \\
&=& \begin{cases} \displaystyle\frac{1}{\mu(t)}\log\left(1+\mu(t)\frac{p^{\Delta}(t)}{p(t)}\right) &\text{for}\; \mu(t)\ne 0 \\ \frac{p^{\Delta}(t)}{p(t)} &\text{for}\; \mu(t)=0. \end{cases}
\end{eqnarray*}
Now substitute $\mu p^{\Delta}=p^{\sigma}-p$. This ends the proof.
\end{proof}

% Example of p(t)=1/t %

\begin{example}\label{example1overt}
Let $t\in\T$ with $t\ne 0$, and set $p(t)=t$. For $s\in\T$, we have
$$ \ell_{p}^{\Delta}(t,s)=\begin{cases} \displaystyle\frac{1}{\mu(t)}\log\left(\frac{\sigma(t)}{t}\right) &\text{for}\; \mu(t)\ne 0 \\ \displaystyle\frac{1}{t} &\text{for}\; \mu(t)=0, \end{cases} $$
where $\Delta$-differentiation is with respect to $t$. Thus,
$$ \ell_{p}^{\Delta}(t,s)
=\begin{cases} 
 \displaystyle\frac{1}{t} &\text{for}\; \T=\R \\
 \displaystyle\frac{1}{h}\log\left(1+\frac{h}{t}\right) &\text{for}\; \T=h\Z \\
 \displaystyle\frac{\log(q)}{(q-1)t} &\text{for}\; \T=q^{\N_{0}},
\end{cases} $$
where $h>0$ and $q>1$.

See Figure 1 for $\T=\Z$. This ends the example.
\begin{figure}
  \centering
    \includegraphics[width=0.6\textwidth]{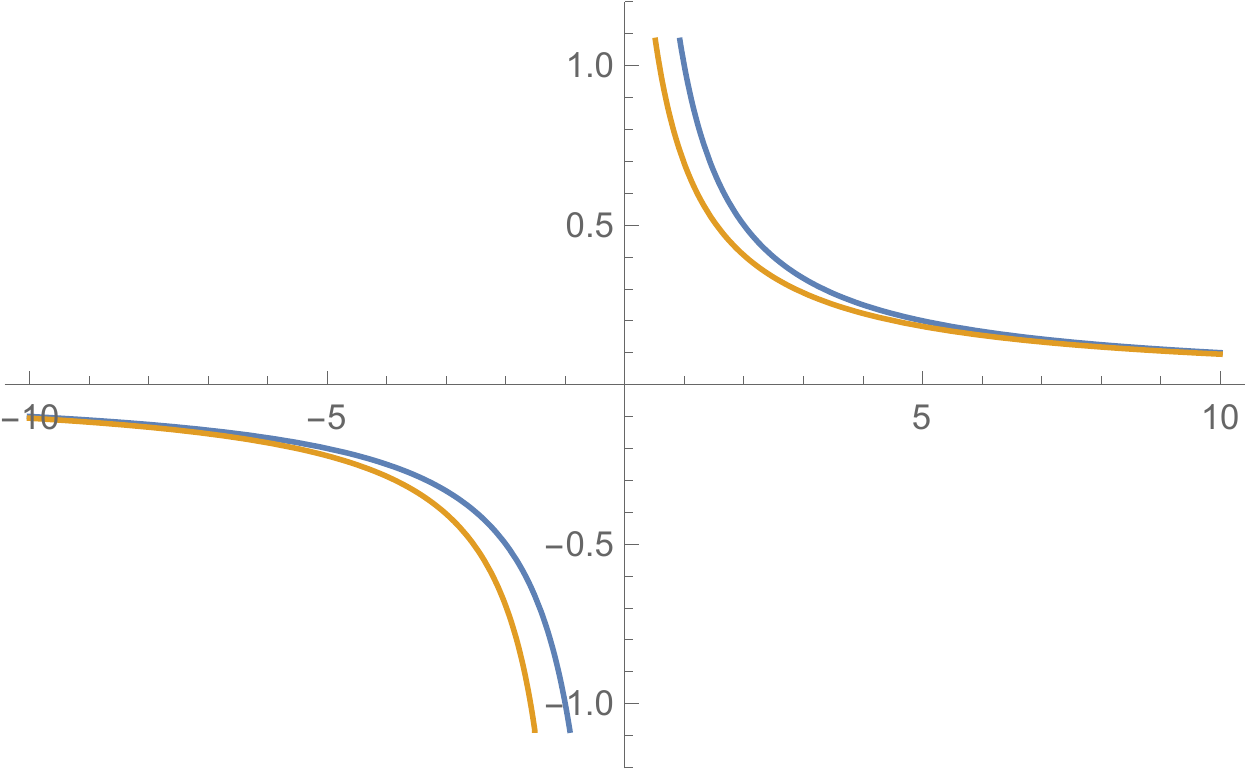}
		  \caption{A plot of the function $\frac{1}{t}$ versus $\Log\left(1+\frac{1}{t}\right)$ for Example \ref{example1overt}.}
\end{figure}
\end{example}

% Example %

\begin{example}
Consider a discrete time scale with alternating graininess function. In particular, for the two step sizes $\alpha,\beta>0$ with $\alpha\ne\beta$, let
\[ \T:=\ensuremath{\mathbb{T}}_{\alpha,\beta}=\{0, \alpha, (\alpha+\beta), (\alpha+\beta)+\alpha, 2(\alpha+\beta), 2(\alpha+\beta)+\alpha, 3(\alpha+\beta),\cdots\}. \]
Then, for $t\in\T$ and $k\in\N_0=\{0,1,2,3,4,\cdots\}$, we have 
\[ \mu(t)=\begin{cases} \alpha &\text{for}\; t=k(\alpha+\beta), \\ \beta &\text{for}\; t=k(\alpha+\beta)+\alpha. \end{cases} \]
Set $p(t)=t$. We claim that for $t\in\ensuremath{\mathbb{T}}_{\alpha,\beta}$ with $t\ne 0$, 
\[ \ell_{p}^{\Delta}(t,s)
=\begin{cases} 
 \displaystyle\frac{1}{\alpha}\log\left(1+\frac{\alpha}{t}\right) &\text{for}\; t = k(\alpha+\beta) \\
 \displaystyle\frac{1}{\beta}\log\left(1+\frac{\beta}{t}\right) &\text{for}\; t = k(\alpha+\beta)+\alpha.
\end{cases} \]
To see this, note that
\begin{eqnarray*}
 \ell_{p}^{\Delta}(t,s)
&=& \displaystyle\frac{1}{\mu(t)}\log\left(\frac{\sigma(t)}{t}\right) \\
&=& \begin{cases} 
 \displaystyle\frac{1}{\alpha}\log\left(\frac{k(\alpha+\beta)+\alpha}{k(\alpha+\beta)}\right) &\text{for}\; t=k(\alpha+\beta) \\
 \displaystyle\frac{1}{\beta}\log\left(\frac{(k+1)(\alpha+\beta)}{k(\alpha+\beta)+\alpha}\right) &\text{for}\; t=k(\alpha+\beta)+\alpha
\end{cases}  \\
&=& \begin{cases} 
 \displaystyle\frac{1}{\alpha}\log\left(1+\frac{\alpha}{t}\right) &\text{for}\; t=k(\alpha+\beta) \\
 \displaystyle\frac{1}{\beta}\log\left(1+\frac{\beta}{t}\right) &\text{for}\; t=k(\alpha+\beta)+\alpha.
\end{cases}
\end{eqnarray*}
This ends the example.
\end{example}

%%%%%%%%%%%%%%%%%%
% Section: nabla %
%%%%%%%%%%%%%%%%%%

\section{the nabla case}\label{nablacase}

A logarithm is also possible for the nabla case.

% Definition: Single Cylinder Transformation %

\begin{definition}[Cylinder Transformation]
For $h>0$, define the single-valued cylinder transformation $\widehat{\xi}_h:\widehat{\C}_h\rightarrow\Z_h$ by
\begin{equation}\label{single-nu-cylinder}
 \widehat{\xi}_h(z) = \begin{cases} \displaystyle\frac{-1}{h}\Log(1-zh) &\text{for}\; h\ne 0 \\ z &\text{for}\; h=0 \end{cases} 
\end{equation}
and the multi-valued cylinder transformation $\widehat{\zeta}_h:\widehat{\C}_h\rightarrow\C$ by
\begin{equation}\label{multi-nu-cylinder}
\widehat{\zeta}_h(z) = \begin{cases} \displaystyle\frac{-1}{h}\log(1-zh) &\text{for}\; h\ne 0 \\ z &\text{for}\; h=0. \end{cases} 
\end{equation}
Here $\C$ is the set of complex numbers, $\Z_h$ is in \eqref{chzhdefs},
$$ \widehat{C}_h=\left\{z\in\C:\; z\ne \frac{1}{h}\right\}, $$
and $\Log$ is again the principal logarithm function.
\end{definition}

The following definition is \cite[Definition 3.4]{bp2}.

% Definition: Regressive Functions %

\begin{definition}[Regressive Function]
A function $p:\T\rightarrow\R$ is $\nu$-regressive provided
\[ 1-\nu(t)p(t) \ne 0 \quad\text{for all}\quad t\in\T_{\kappa} \]
holds. The set of all $\nu$-regressive and ld-continuous functions $p:\T\rightarrow\R$ is denoted by $\widehat{\mathcal{R}}$.
\end{definition}

The following definition is \cite[Definition 3.10]{bp2}.

% Definition: Exponential Function %

\begin{definition}[Exponential Function]
For functions $p\in\widehat{\mathcal{R}}$, the nabla exponential function on time scales is given by
\[ \widehat{e}_p(t,s)=\exp\left(\int_s^t \widehat{\xi}_{\nu(\tau)}(p(\tau))\nabla\tau\right) \quad\text{for}\quad s,t\in\T, \]
where $\widehat{\xi}_h(z)$ is the single-valued cylinder transformation given in \eqref{single-nu-cylinder}.
\end{definition}

We now offer a new definition of logarithms for the nabla case on time scales.

% Definition: Logarithm Function %

\begin{definition}[Logarithm Function]
For a $\nabla$-differentiable function $p:\T\rightarrow\R$ with $p\ne 0$ on $\T$, the multi-valued nabla logarithm function on time scales is given by
\[ \widehat{\ell}_p(t,s) = \int_s^t \widehat{\zeta}_{\nu(\tau)}\left(\frac{p^{\nabla}(\tau)}{p(\tau)}\right)\nabla\tau \quad\text{for}\quad s,t\in\T, \]
where $\widehat{\zeta}_h(z)$ is the multi-valued cylinder transformation given in \eqref{multi-nu-cylinder}, while the principal nabla logarithm is given by
\[ \widehat{L}_p(t,s) = \int_s^t \widehat{\xi}_{\nu(\tau)}\left(\frac{p^{\nabla}(\tau)}{p(\tau)}\right)\nabla\tau \quad\text{for}\quad s,t\in\T, \]
where $\widehat{\xi}_h(z)$ is the single-valued nabla cylinder transformation given in \eqref{single-nu-cylinder}
\end{definition}

Properties analogous to those given earlier can be established for the nabla case as well.

%%%%%%%%%%%%%%%%%%
%                %
% Cayley Section %
%                %
%%%%%%%%%%%%%%%%%%

\section{logarithms for Cayley-exponential functions}\label{SectionCayley}

In \cite{cieslinski}, the author introduced an improved exponential function (or the Cayley-exponential function) on a time scale defined by
\begin{equation}\label{cayley-exp}
 E_p(t,s) = \exp\left(\int_s^t \Psi_{\mu(\tau)}(p(\tau))\Delta\tau\right), 
\end{equation}
where $p:\T\rightarrow\C$ is rd-continuous and satisfies the regressivity condition $\mu(\tau)p(\tau)\ne \pm 2$ for all $\tau\in\T^{\kappa}$, and the modified cylinder transformation $\Psi$ is given by
\begin{equation}\label{cayley-singlecyl}
 \Psi_h(z) = \frac{1}{h}\Log\left(\frac{1+\frac{1}{2}zh}{1-\frac{1}{2}zh}\right), \quad \Psi_0(z)=z,
\end{equation}
for $h>0$. Here again, $\Log$ represents the principal complex logarithm.
Consider the multi-valued function version of \eqref{cayley-singlecyl} denoted, i.e., 
\begin{equation}\label{cayley-multicyl}
\psi_h(z) = \frac{1}{h}\log\left(\frac{1+\frac{1}{2}zh}{1-\frac{1}{2}zh}\right), \quad \psi_0(z)=z,
\end{equation}
where $\log$ represents the multi-valued complex logarithm. We introduce the following Cayley-logarithm functions on time scales.

% Definition: Cayley-logarithms %

\begin{definition}
For a $\Delta$-differentiable function $p:\T\rightarrow\C$ with $p\ne 0$ on $\T$, the multi-valued Cayley-logarithm function on time scales is given by
\[ \operatorname{caylog}_p(t,s) = \int_s^t \psi_{\mu(\tau)}\left(\frac{2p^{\Delta}(\tau)}{p(\tau)+p^{\sigma}(\tau)}\right)\Delta\tau \quad\text{for}\quad s,t\in\T, \]
where $\psi_h(z)$ is the multi-valued cylinder transformation given in \eqref{cayley-multicyl}. Define the principal Cayley-logarithm on time scales to be
\[ \operatorname{CayLog}_p(t,s) = \int_s^t \Psi_{\mu(\tau)}\left(\frac{2p^{\Delta}(\tau)}{p(\tau)+p^{\sigma}(\tau)}\right)\Delta\tau \quad\text{for}\quad s,t\in\T, \]
where $\Psi_h(z)$ is the single-valued cylinder transformation given in \eqref{cayley-singlecyl}.
\end{definition}

% Lemma %

\begin{lemma}
The Cayley-logarithm functions are well-defined functions.
\end{lemma}

\begin{proof}
For a $\Delta$-differentiable function $p:\T\rightarrow\C$ with $p\ne 0$ on $\T$, we need to show that 
\[ \mu(\tau)\frac{2p^{\Delta}(\tau)}{p(\tau)+p^{\sigma}(\tau)} \ne \pm 2, \]
in other words, that the regressivity condition holds. The following are equivalent:
\begin{eqnarray*}
\frac{2\mu(\tau)p^{\Delta}(\tau)}{p(\tau)+p^{\sigma}(\tau)} &=& \pm 2 \\
\frac{p^{\sigma}(\tau)-p(\tau)}{p(\tau)+p^{\sigma}(\tau)} &=& \pm 1 \\
p^{\sigma}(\tau)-p(\tau) &=& \pm\left(p(\tau)+p^{\sigma}(\tau)\right) \\
p^{\sigma}(\tau) \mp p^{\sigma}(\tau) &=& p(\tau) \pm p(\tau),
\end{eqnarray*}
so that we have either  $0=2p(\tau)$ or $2p^{\sigma}(\tau)=0$, both contradictions.
\end{proof}

\begin{theorem}
For a $\Delta$-differentiable function $p:\T\rightarrow\C$ with $p\ne 0$ on $\T$,
\begin{equation}
 \operatorname{caylog}_p(t,s) = \ell_p(t,s) \quad\text{and}\quad \operatorname{CayLog}_p(t,s) = L_p(t,s)
\end{equation}
for all $t,s\in\T$.
\end{theorem}

\begin{proof}
Consider \eqref{cayley-multicyl}. For fixed $\tau\in\T^{\kappa}$ with $\mu(\tau)\ne 0$, we have
\begin{eqnarray*}
\psi_{\mu(\tau)}\left(\frac{2p^{\Delta}(\tau)}{p(\tau)+p^{\sigma}(\tau)}\right) &=&  \frac{1}{\mu(\tau)} \log\left(\frac{1+\frac{1}{2}\frac{2p^{\Delta}(\tau)}{p(\tau)+p^{\sigma}(\tau)}\mu(\tau)}{1-\frac{1}{2}\frac{2p^{\Delta}(\tau)}{p(\tau)+p^{\sigma}(\tau)}\mu(\tau)}\right) \\
&=& \frac{1}{\mu(\tau)} \log\left(\frac{1+\frac{\mu(\tau)p^{\Delta}(\tau)}{p(\tau)+p^{\sigma}(\tau)}}{1-\frac{\mu(\tau)p^{\Delta}(\tau)}{p(\tau)+p^{\sigma}(\tau)}}\right) \\
&=& \frac{1}{\mu(\tau)} \log\left(\frac{1+\frac{p^{\sigma}(\tau)-p(\tau)}{p(\tau)+p^{\sigma}(\tau)}}{1-\frac{p^{\sigma}(\tau)-p(\tau)}{p(\tau)+p^{\sigma}(\tau)}}\right) \\
&=& \frac{1}{\mu(\tau)} \log\left(\frac{p^{\sigma}(\tau)}{p(\tau)}\right) \\
&=& \frac{1}{\mu(\tau)} \log\left(\frac{p(\tau)+\mu(\tau)p^{\Delta}(\tau)}{p(\tau)}\right) \\
&=& \zeta_{\mu(\tau)}\left(\frac{p^{\Delta}(\tau)}{p(\tau)}\right)
\end{eqnarray*}
for $\zeta_h$ defined in \eqref{cylinder}. For fixed $\tau\in\T^{\kappa}$ with $\mu(\tau)= 0$, we have $\sigma(\tau)=\tau$ and
$$ \frac{2p^{\Delta}(\tau)}{p(\tau)+p^{\sigma}(\tau)} = \frac{p^{\Delta}(\tau)}{p(\tau)}. $$ 
Consequently,
\begin{eqnarray*}
\psi_{\mu(\tau)}\left(\frac{2p^{\Delta}(\tau)}{p(\tau)+p^{\sigma}(\tau)}\right) &=&  
\frac{2p^{\Delta}(\tau)}{p(\tau)+p^{\sigma}(\tau)} \\
&=&\frac{p^{\Delta}(\tau)}{p(\tau)} \\
&=& \zeta_{\mu(\tau)}\left(\frac{p^{\Delta}(\tau)}{p(\tau)}\right).
\end{eqnarray*}
Thus, in either case, we have 
$$ \psi_{\mu(\tau)}\left(\frac{2p^{\Delta}(\tau)}{p(\tau)+p^{\sigma}(\tau)}\right)= \zeta_{\mu(\tau)}\left(\frac{p^{\Delta}(\tau)}{p(\tau)}\right). $$
It follows that
\begin{eqnarray*}
\operatorname{caylog}_p(t,s) &=& \int_s^t \psi_{\mu(\tau)}\left(\frac{2p^{\Delta}(\tau)}{p(\tau)+p^{\sigma}(\tau)}\right)\Delta\tau \\
&=& \int_s^t \zeta_{\mu(\tau)}\left(\frac{p^{\Delta}(\tau)}{p(\tau)}\right)\Delta\tau \\
&=& \ell_p(t,s).
\end{eqnarray*}
Similarly, we have
$$ \operatorname{CayLog}_p(t,s) = L_p(t,s), $$
completing the proof. 
\end{proof}

% Remark %

\begin{remark}
The previous theorem and proof can be generalized in the following way. Let $\eta\in[0,1]$, and set
\begin{equation}
\psi^{\eta}_h(z) = \frac{1}{h}\log\left(\frac{1+(1-\eta)hz}{1-\eta hz}\right), \quad \psi^{\eta}_0(z)=z.
\end{equation}
Then, for a $\Delta$-differentiable function $p:\T\rightarrow\C$ with $p\ne 0$ on $\T$, and for all $\tau\in\T^\kappa$, we have 
\begin{eqnarray*}
\lefteqn{\psi^{\eta}_{\mu(\tau)}\left(\frac{p^{\Delta}(\tau)}{(1-\eta)p(\tau)+\eta p^{\sigma}(\tau)}\right)}\\ 
&=& \frac{1}{\mu(\tau)} \log\left(\frac{1+(1-\eta)\mu(\tau)\frac{p^{\Delta}(\tau)}{(1-\eta)p(\tau)+\eta p^{\sigma}(\tau)}}{1-\eta\mu(\tau)\frac{p^{\Delta}(\tau)}{(1-\eta)p(\tau)+\eta p^{\sigma}(\tau)}}\right) \\
&=& \frac{1}{\mu(\tau)} \log\left(\frac{(1-\eta)p(\tau)+\eta p^{\sigma}(\tau)+(1-\eta)\mu(\tau)p^{\Delta}(\tau)}{(1-\eta)p(\tau)+\eta p^{\sigma}(\tau)-\eta\mu(\tau)p^{\Delta}(\tau)}\right) \\
&=& \frac{1}{\mu(\tau)} \log\left(\frac{(1-\eta)p(\tau)+\eta p^{\sigma}(\tau)+(1-\eta)(p^\sigma(\tau)-p(\tau))}{(1-\eta)p(\tau)+\eta p^{\sigma}(\tau)-\eta(p^\sigma(\tau)-p(\tau))}\right) \\
&=& \frac{1}{\mu(\tau)} \log\left(\frac{p^{\sigma}(\tau)}{p(\tau)}\right) \\
&=& \frac{1}{\mu(\tau)} \log\left(\frac{p(\tau)+\mu(\tau)p^{\Delta}(\tau)}{p(\tau)}\right) \\
&=& \zeta_{\mu(\tau)}\left(\frac{p^{\Delta}(\tau)}{p(\tau)}\right)
\end{eqnarray*}
for $\zeta_h$ defined in \eqref{cylinder}. For fixed $\tau\in\T^{\kappa}$ with $\mu(\tau)= 0$, we have $\sigma(\tau)=\tau$ and
$$ \frac{p^{\Delta}(\tau)}{(1-\eta)p(\tau)+\eta p^{\sigma}(\tau)} = \frac{p^{\Delta}(\tau)}{p(\tau)}. $$ 
As a result,
\begin{eqnarray*}
\psi^{\eta}_{0}\left(\frac{p^{\Delta}(\tau)}{(1-\eta)p(\tau)+\eta p^{\sigma}(\tau)}\right) &=&  
\frac{p^{\Delta}(\tau)}{(1-\eta)p(\tau)+\eta p^{\sigma}(\tau)} \\
&=&\frac{p^{\Delta}(\tau)}{p(\tau)} \\
&=& \zeta_{0}\left(\frac{p^{\Delta}(\tau)}{p(\tau)}\right).
\end{eqnarray*}
Thus, in either case, we have 
$$ \psi^{\eta}_{\mu(\tau)}\left(\frac{p^{\Delta}(\tau)}{(1-\eta)p(\tau)+\eta p^{\sigma}(\tau)}\right)= \zeta_{\mu(\tau)}\left(\frac{p^{\Delta}(\tau)}{p(\tau)}\right) $$
for all $\eta\in[0,1]$. Consequently,
\begin{eqnarray*}
\log^{\eta}_p(t,s) &:=& \int_s^t \psi^{\eta}_{\mu(\tau)}\left(\frac{p^{\Delta}(\tau)}{(1-\eta)p(\tau)+\eta p^{\sigma}(\tau)}\right)\Delta\tau \\
&=& \int_s^t \zeta_{\mu(\tau)}\left(\frac{p^{\Delta}(\tau)}{p(\tau)}\right)\Delta\tau \\
&=& \ell_p(t,s).
\end{eqnarray*}
This ends the remark.
\end{remark}

%%%%%%%%%%%%%%%%%%
% History of log %
%%%%%%%%%%%%%%%%%%

\section{previous logarithms on time scales}\label{historylog}

As shown in previous sections, the key to arriving at useful logarithm properties is to allow for a multi-valued logarithm, as exists for the $\T=\R$ case. Here, we present the previous definitions of a logarithm on time scales, noting that they are all single-valued functions.

The first logarithm on time scales \cite{huff} interprets the integral
\[ \int_{t_0}^t \frac{2}{\tau+\sigma(\tau)}\Delta\tau \]
as a time scales analogue of $\ln t$. This is understandable, because if $\T=\R$, then $\sigma(\tau)=\tau$, and
\[ \int_{t_0}^t \frac{2}{\tau+\sigma(\tau)}\Delta\tau = \int_{t_0}^t \frac{2}{2\tau} d\tau = \ln t - \ln t_0. \]

A second approach \cite[Section 3]{bohner} is to view the slightly different integral
\[ \int_{t_0}^t \frac{1}{\tau+2\mu(\tau)}\Delta\tau \] 
as the time scales version of $\ln t$, due to the same fact that it reduces to $\ln t - \ln t_0$ on $\T=\R$, and as it is part of a solution form to a certain Euler–Cauchy dynamic equation whose differential equation analogue involves the natural logarithm. 

A third approach \cite[Section 4]{bohner} could be to define a logarithm via
\[ L_p(t,t_0)=\int_{t_0}^{t}\frac{p^{\Delta}(\tau)}{p(\tau)}\Delta\tau \]
for $\Delta$-differentiable functions $p:\T\rightarrow\R$. Clearly if $p(\tau)=\tau$, then this is
\[ L_p(t,t_0)=\int_{t_0}^{t}\frac{p^{\Delta}(\tau)}{p(\tau)}\Delta\tau = \int_{t_0}^{t}\frac{1}{\tau}\Delta\tau, \]
a form that is similar to its continuous analogue for $\T=\R$.

A fourth approach \cite{jackson} is to take the logarithm to be given by
\[ \log_{\T} p(t)=\frac{p^{\Delta}(t)}{p(t)} \]
for $\Delta$-differentiable functions $p:\T\rightarrow\R$, where the time scale logarithm on $\R$ does not play the role of the logarithm, clearly, but rather its derivative. The motivation here is to maintain some attractive algebraic properties of logarithms, and to serve in some sense as an inverse to the exponential function.

A fifth approach \cite{mozyrska}, only for time scales such that $1\in\T$, is to define the natural logarithm via
\[ L_{\T}(t)=\int_1^t\frac{1}{\tau}\Delta\tau, \]
which hearkens back to \cite[Section 4]{bohner}. Here the motivation is clearly that
\[ L_{\R}(t)=\ln t, \quad L_{\T}(1)=0, \quad L^{\Delta}_{\T}(t)=\frac{1}{t}. \] 

Each of these definitions has advantages and drawbacks, and each one satisfies some of what one might wish for in a logarithm function. As shown earlier in this work, however, a multi-valued logarithm on time scales with a definition based on cylinder transformations is a natural move that leads to nice properties.

%\section*{Acknowledgements} 

%%%%%%%%%%%%%%%% 
% Bibliography % 
%%%%%%%%%%%%%%%%

\end{document}